\documentclass[a4paper]{amsart}
\usepackage[english]{babel}
\usepackage{lmodern}
\usepackage{newlfont}
\usepackage{booktabs}
\usepackage{color}
\usepackage[T1]{fontenc}
\usepackage[utf8]{inputenc}
\usepackage{indentfirst}
\usepackage{amsmath}
\usepackage{amsmath,amssymb}
\usepackage{amsthm}
\usepackage{geometry}
\usepackage{mathtools}
\usepackage{listings}
\usepackage{amsfonts}
\usepackage{braket}
\usepackage{emptypage}
\usepackage{newlfont}
\usepackage{amssymb}
\usepackage{graphicx}
\usepackage[italian]{varioref}
\usepackage{accents}
\usepackage{comment}
\usepackage{stmaryrd}

\makeatletter
\labelformat{equation}{\tagform@{#1}}
\makeatother
\usepackage{hyperref}
\usepackage{relsize}
\usepackage{faktor}
\usepackage{enumerate}
\usepackage{fancyhdr}

\usepackage[colorinlistoftodos]{todonotes}

\DeclarePairedDelimiter{\abs}{\lvert}{\rvert}

\renewcommand{\phi}{\varphi}
\renewcommand{\epsilon}{\varepsilon}

\newcommand{\la}{\lambda}
\newcommand{\numberset}{\mathbb}
\newcommand{\e}{\varepsilon}
\newcommand{\eps}{\varepsilon}
\newcommand{\N}{\numberset{N}}
\newcommand{\C}{\numberset{C}}
\newcommand{\R}{\numberset{R}}

\newcommand{\norm}[1]{\left\|#1\right\|}

\newcommand{\si}{0}
\newcommand{\Om}{\Omega}

\theoremstyle{definition}

\theoremstyle{definition}                                                                         
\newtheorem{definizione}{Definizione}[section]
\theoremstyle{definition}                                                                         
\newtheorem{remark}[definizione]{Remark}
\theoremstyle{plain}                                                                              
\newtheorem{thm}[definizione]{Theorem}
\theoremstyle{plain}    

\theoremstyle{plain}     
\newtheorem{lemma}[definizione]{Lemma}
\theoremstyle{plain}

\theoremstyle{definition}                                                                         

\numberwithin{equation}{section}

\begin{document}
\title[On the number of critical points of stable solutions]{On the number of critical points of stable solutions in bounded strip-like domains}
\thanks{This work was supported by INDAM-GNAMPA}

\author[De Regibus]{Fabio De Regibus}
\address{Dipartimento di Matematica, Universit\`a di Roma ``La Sapienza'', P.le A. Moro 2 - 00185 Roma, Italy, e-mail: {\sf fabio.deregibus@uniroma1.it}.}
\author[Grossi]{Massimo Grossi }
\address{Dipartimento di Matematica, Universit\`a di Roma ``La Sapienza'', P.le A. Moro 2 - 00185 Roma, Italy, e-mail: {\sf massimo.grossi@uniroma1.it}.}

\maketitle
\begin{abstract}
In this paper we show that there exists a family of domains $\Om_\e\subseteq\R^N$ with $N\ge2$, such that the \emph{stable} solution of the problem
\[
\begin{cases}
-\Delta u= g(u)&\hbox{in }\Omega_\e\\
u>0&\hbox{in }\Omega_\e\\
u=0&\hbox{on }\partial\Omega_\e
\end{cases}
\]
admits $k$ critical points with $k\ge2$. Moreover the sets $\Om_\e's$ are star-shaped and ``close'' to a strip as $\e\to0$. Next, if $g(u)\equiv1$ and $N\ge3$ we exhibit a family of domain $\Om_\e's$ with {\em positive mean curvature} and solutions $u_\e$ which have $k$ critical points with $k\ge2$. In this case, the domains $\Om_\e$ turn out to be ``close'' to a cylinder as $\e\to0$.
\end{abstract}

\section{Introduction and main results}
In this paper we investigate the number of critical points of solutions $u$ to the following problem
\begin{equation}
\label{PB1}
\begin{cases}
-\Delta u=g(u)&\hbox{in }\Omega\\
u>0&\hbox{in }\Omega\\
u=0&\hbox{on }\partial\Omega
\end{cases}
\end{equation}
where $\Omega$ is a smooth bounded domain in $\R^N$, $N\ge2$ and $g$ is a smooth nonlinearity.

It is known that this problem strongly depends on the geometry of $\Om$. A very studied case is when $\Om$ is \emph{convex}: in this case it is expected the uniqueness of the critical point. One of the first results in this direction is the one in~\cite{ml}, where it is proved the strict convexity of the level sets for $N=2$ and $g\equiv1$. Of course this property implies the uniqueness of the critical point.

Another classical problem concerns the first eigenfunction of the laplacian with zero Dirichlet boundary condition. In this case the uniqueness of the critical point was proved in~\cite{bl} (see also~\cite{app}).

A very general result on the uniqueness of the critical point of solutions of~\ref{PB1} is given in the seminal paper by~\cite{gnn} where it is only assumed that $g$ is a locally Lipschitz function and $\Om$ is a symmetric domain in $\R^N$ which is convex in any direction.

Some conjectures claim that the symmetry assumption in Gidas, Ni, Nirenberg's Theorem can be removed. An interesting contribution in this direction is the result in~\cite{cc} where the uniqueness of the critical point is proved for \emph{semi-stable} solutions in planar domains with strictly positive boundary curvature.

We recall that a solution $u$ of problem~\ref{PB1} is said to be \emph{stable} (or \emph{semi-stable}) if the linearized operator at $u$ is positive definite, i.e. if for all $\xi\in\mathcal C^{\infty}_{0}(\Omega)\setminus\{0\}$ one has
\[
\int_{\Omega}|\nabla\xi|^{2}-\int_{\Omega}g'(u)|\xi|^{2}>0\ (\ge0),
\]
or equivalently if the first eigenvalue of the linearized operator $-\Delta-g'(u)$ in $\Omega$ is positive (non-negative).

The result in~\cite{cc} was recently extended allowing $\partial\Om$ to have points with zero curvature, see~\cite{dgm}.

Next we are going to discuss what happens if $\partial\Om$ contains points with \emph{negative} curvature. We will see that not only the uniqueness of the critical point is lost, but it is not even possible to have any bound on the number of critical points. Indeed, in~\cite{gg3} it was proved that there exists a family of bounded domains $\Om_\e$ in $\R^2$ and a solution $u_\e$ to 
\[
\begin{cases}
-\Delta u_\e=1&\hbox{in }\Omega_\e\\
u_\e>0&\hbox{in }\Omega_\e\\
u_\e=0&\hbox{on }\partial\Omega_\e
\end{cases}
\]
such that
\begin{enumerate}[(i)]
\item $\Omega_\e$ is starshaped with respect to an interior point;
\item $\Omega_\e$ locally converges to a strip $\mathcal S=\set{(x,y)\in\R^2|-1<y<1}$ for $\e\to0$, i.e. for all compact set $K\subseteq\R^2$ it holds $\abs{K\cap(\mathcal S\Delta\Omega_\e)}\to0$ as $\e\to0$;
\item The curvature of $\partial\Omega_\e$ change sign once and $\min\limits_{(x,y)\in\partial\Om_\e} Curv(x,y)\to0$ as $\e\to0$;
\item $u_\e$ has at least $k$ maximum points with $k\ge2$.\\
\end{enumerate}
In some sense, for $\e>0$ small, the domains $\Omega_\e$ are ``close'' to be convex and the minimum negative value of the curvature of $\partial\Omega_\e$ is close to zero as we want.

The aim of this paper is twofold: first we want to extend the result of~\cite{gg3} to more general nonlinearities.  On the other hand we want to investigate the role of the curvature of $\partial\Om$ in higher dimensions.

Concerning the first point, let us assume that the nonlinearity has the form $g=\lambda f$ where $f$ is smooth and satisfies
\begin{equation}
\label{i2}
f:\R\to\R\,\hbox{ is increasing and convex},
\end{equation}
\begin{equation}\label{i3}
f(0)>0.
\end{equation}
In this setting it is well known that there exists $\lambda^*(\Omega)>0$ such that for all $\lambda\in (0,\lambda^*(\Omega))$ the problem
\begin{equation}
\label{PB0}
\begin{cases}
-\Delta u=\lambda f(u)&\hbox{in }\Omega\\
u>0&\hbox{in }\Omega\\
u=0&\hbox{on }\partial\Omega
\end{cases}
\end{equation}
admits a positive stable solution, see for instance~\cite{Bandle(book)},~\cite{cr75} and~\cite{mp} and the references therein.

Finally let us denote by $\mathcal S$ the strip $\mathcal S=\set{(x,y)\in\R^{N}\times\R| -1<y<1}$. Our first result claims that, if $f$ satisfies~\ref{i2} and~\ref{i3} then there exists a family of bounded smooth domains $\Om_\e$ ``close'' to the strip $\mathcal S$ and a solution $u_\e$ to~\ref{PB0} with $k$ maximum points, $k\ge2$. The precise statement follows.
\begin{thm}
\label{THM1}
Assume that $f$ satisfies~\ref{i2} and~\ref{i3}.

Then for any $\lambda\in(0,\la^*(-1,1))$ and for all $k\in\N$ there exists a family of smooth and bounded domain $\Omega_\e\subseteq\R^{N+1}$ such that
\begin{enumerate}[(i)]
\item $\Omega_\e$ is starshaped with respect to the origin and symmetric with respect to the hyperplanes $x_j=0$ for $j=1,\dots,N$ and $y=0$;
\item $\Omega_\e$ locally converges to the strip $\mathcal S$ for $\e\to0$, i.e. for all compact set $K\subseteq\R^{N+1}$ it holds $\abs{K\cap(\mathcal S\Delta\Omega_\e)}\to0$ as $\e\to0$;
\item $\lambda^*(\Omega_\e)\ge\la^*(-1,1)$ for $\e$ small enough;
\item if $u_\e$ is the stable solution of problem~\ref{PB0} in $\Omega_\e$ for some $0<\lambda<\lambda^*(\Omega_\e)$ then $u_\e$ has at least $k$ maximum points.
\end{enumerate}
\end{thm}
Let us give an idea of Theorem~\ref{THM1}. The assumptions on $f$ imply that there exists a \emph{stable} solution $u_0$ of the following ODE
\[
\begin{cases}
-u''=\lambda f(u)&\hbox{in }(-1,1)\\
u>0&\hbox{in }(-1,1)\\
u(\pm1)=0.
\end{cases}
\]
Next, for a small $\sigma>0$ let us extend $u_0$ to a slightly larger interval $(-1-\sigma,1+\sigma)$ and denote by $\phi:\R^{N+1}\to\R$ a suitable solution of the following PDE
\begin{equation}\label{i4}
-\Delta v=\lambda f'(u_\si(y))v,\quad\text{in }\R^N\times\left(-1-\sigma ,1+\sigma \right).
\end{equation}
Of course~\ref{i4} can be solved using the classical separation of variables method. 

Our domain $\Om_\e$ will be  the connected component of $\set{u_\si+\e\phi>0}$ containing the origin and the solution $u_\e$ the \emph{stable} solution to~\ref{PB0} with $\Om=\Om_\e$. Finally we show that $u_\e$ is close to $u_0+\e\phi$ on the compact sets of $\Omega_\e$ and, since it will be proved that this last function admits $k$ nondegenerate critical points then $(iv)$ follows.

The proof of Theorem~\ref{THM1} will be given in Section~\ref{SEZ2}. We point out that it is possible to prove a slight more general result for problem~\ref{PB1} without assuming~\ref{i3}, see Remark~\ref{rmk:general}.

It is important to remark that our construction only works for {\em stable} solutions to~\ref{PB1}. 
 Indeed, even for the case of the first eigenfunction of the laplacian (where the first eigenvalue of the linearized problem is \emph{zero}), we are not able to construct a domain $\Om_\e$ as in Theorem~\ref{THM1}. This will be discussed in Remark~\ref{R}.
We do not know if in this case there exists a pair $(\Om_\e,u_\e)$ as in Theorem~\ref{THM1}.

Next let us discuss the role of the curvature of $\partial\Om$ for solutions to~\ref{PB1} in higher dimensions. We will focus on the particular case of the torsion problem, i.e. $g\equiv1$ in~\ref{PB1}. By Makar-Limanov's result if $N=2$ and the curvature of $\partial\Om$ is positive then the solution $u$ admits exactly one critical point (see~\cite{dgm} if the curvature vanishes somewhere). So a question naturally arises:

{\em if $N\ge3$ what is a sufficient condition on $\partial\Om$ which implies the uniqueness of the critical point?}

We point out that even for the torsion problem this is an open problem. In the second part of this paper we give a contribution to this question showing that the  {\em positive  mean curvature} of $\partial\Om$ is not the correct extension to higher dimensions.

Indeed, for any $k\ge2$, we will construct a domain $\Omega_\e\subseteq\R^N$ with $N\ge3$ and positive mean curvature and a solution $u_\e$ of the torsion problem in $\Omega_\e$ such that $u_\e$ has at least $k$ critical points. Actually we suspect that the correct condition which implies the uniqueness of the critical point for the solution of the torsion problem is that all {\em principal curvatures} are positive. However we have no result to support this idea.

The construction of the pair $(\Om_\e,u_\e)$ is similar to the one in Theorem~\ref{THM1}, but $\Om_\e$ turns to be a suitable perturbation of a cylinder $\mathcal C=\set{(x,y)\in\R\times\R^{N}| \abs{y}^2<1}$ for $N\ge2$. The result is the following,
\begin{thm}
\label{THM2}
Let $N\ge2$. For any $k\in\N$ there exists a family of smooth and bounded domain $\Omega_\e\subseteq\R^{N+1}$ and smooth positive functions $u_\e:\Omega_\e\to\R$ such that
\begin{enumerate}[(i)]
\item $\Omega_\e$ is starshaped with respect to an interior point;
\item $\Omega_\e$ locally converges to the cylinder $\mathcal C$ for $\e\to0$, i.e. for all compact set $K\subseteq\R^{N+1}$ it holds $\abs{K\cap(\mathcal C\Delta\Omega_\e)}\to0$ as $\e\to0$;
\item the mean curvature of $\partial\Omega_\e$ is positve;
\item $u_\e$ solves the torsion problem
\[
\begin{cases}
-\Delta u=1&\hbox{in }\Omega_\e\\
u=0&\hbox{on }\partial\Omega_\e;
\end{cases}
\]
\item $u_\e$ has at least $k$ nondegenerate maximum points.
\end{enumerate}
\end{thm}
\begin{figure}[ht]
\centerline {\includegraphics[width=10cm]{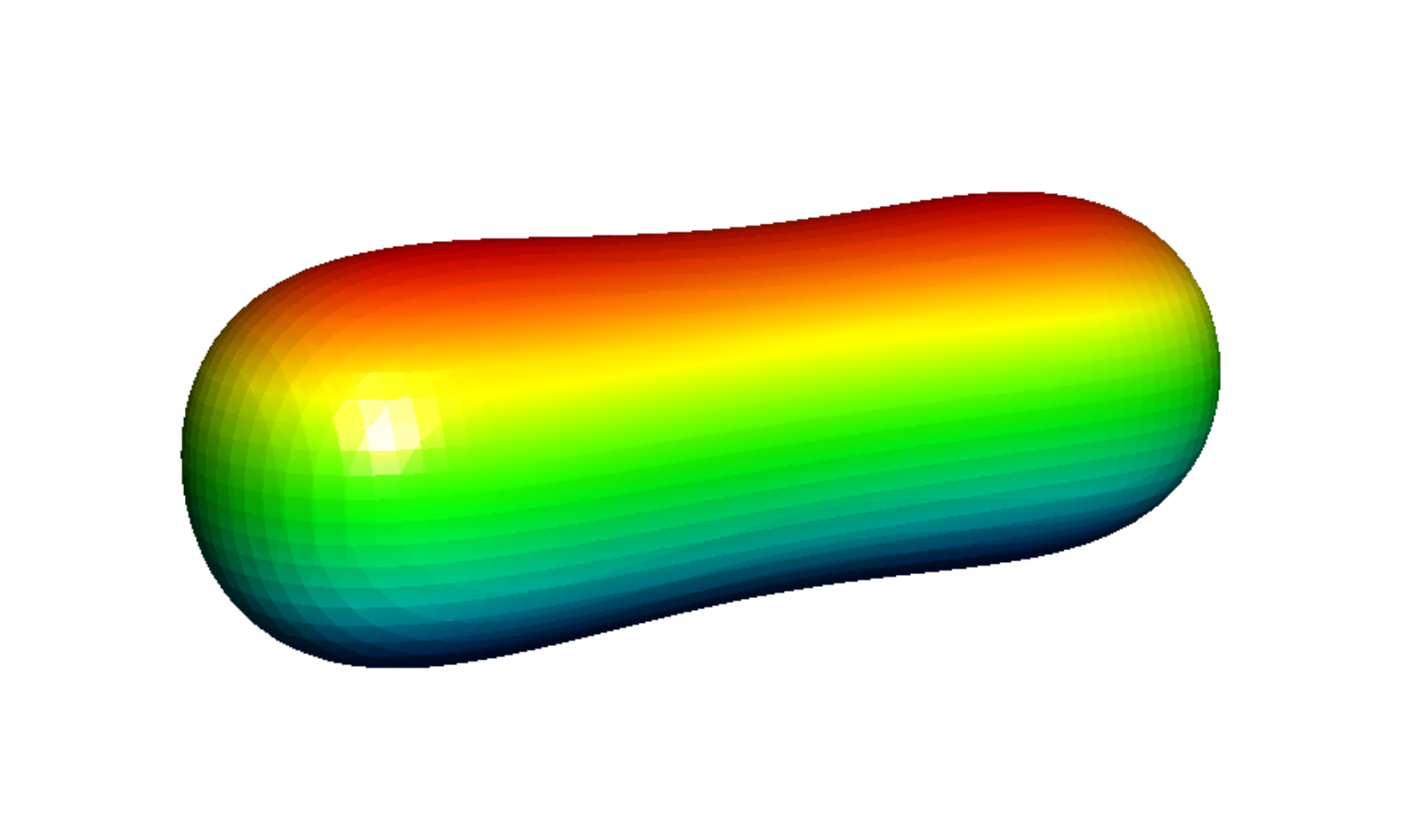}}
\caption{The domain $\Omega_\e$ in Theorem~\ref{THM2}  for $N=2$ and $k=2$.}
\end{figure}
As in Theorem~\ref{THM1} we have that $u_\e=u_0+\e\phi$ where $u_0=\frac{1}{2N}\left(N-\abs{y}^2\right)$ is a solution of the torsion problem in the unit ball in $\R^N$ and $\phi$ turns to be an harmonic function in the whole $\R^{N+1}$.

Then we take $\Omega_\e$ as in Theorem~\ref{THM1}, while our solution will directly be $u_\e=u_0+\e\phi$.
Since the set $\Omega_\e$ turns out to be a small perturbation of the cylinder $\mathcal C$, which boundary has positive mean curvature, then $(iii)$ of Theorem~\ref{THM2} follows. Note that, unlike as in Theorem~\ref{THM1}, here the pair $(\Om_\e,u_\e)$ is explicitly computed.

Theorem~\ref{THM2} will be proved in Section~\ref{SEZ4}. Finally the Appendix is devoted to the detailed proof of some claims in Section~\ref{SEZ2} and Section~\ref{SEZ4}.

\section{Proof of Theorem~\ref{THM1}}
\label{SEZ2}
In this section we take $x=(x_1,\dots,x_N)\in\R^N$ and $y\in\R$ and we assume the hypothesis of Theorem~\ref{THM1}.
The proof works as follows: we first construct a suitable domain $\Omega_\e$ which verifies the claim of Theorem \ref{THM1} and next we prove that the stable solution of~\ref{PB0} satisfies the claim $(iv)$ in Theorem~\ref{THM1}.

The first step in the construction of the domain $\Omega_\e$ is to introduce a solution $u_\si$ of the $1$-dimensional problem
\begin{equation}
\label{PB2}
\begin{cases}
-u''=\lambda f(u)&\hbox{in }(-1,1)\\
u>0&\hbox{in }(-1,1)\\
u(\pm1)=0.
\end{cases}
\end{equation}

By the assumption on $f$ such a solution exists and by elementary argument it can be extended to verify 
\[
\begin{cases}
-u''=\lambda f(u)&\hbox{in }(-1-\sigma,1+\sigma)\\
u>0&\hbox{in }(-1,1)\\
u(\pm1)=0\\
u<0&\hbox{in }[-1-\sigma,-1)\cup(1,1+\sigma]\\
\end{cases}
\]
for $\sigma>0$ and small. We again denote by $u_0$ this extension.

Since $u_0$ is a stable solution we have that the first eigenvalue of the linearized operator
\begin{equation}
\label{eq:eigenv}
-\frac{\mathrm{d}^2}{\mathrm{d}y^2}-\lambda f'(u_0(y)),
\end{equation}
in $(-1,1)$ with Dirichlet boundary conditions is strictly positive. Then, up to choose a smaller $\sigma$, also the first eigenvalue of~\ref{eq:eigenv} in $(-1-\sigma,1+\sigma)$ is strictly positive. We denote it by $\mu_0$.

Next ingredient in the construction of $\Omega_\e$ involves a
solution of a suitable linearized problem in the strip $\R^N\times(-1-\sigma,1+\sigma)$. To do this we need to study the following ODE.
\begin{lemma}
\label{lemma:costruzione:omegai}
For $\mu\in(0,\mu_0)$ there exists a solution $\omega_\mu$ of the ordinary equation
\[
\begin{cases}
-\omega''-\lambda f'(u_\si(y))\omega=\mu\omega\quad\text{in }(-1-\sigma,1+\sigma)\\
\omega_\mu(0)=1
\end{cases}
\]
such that
\begin{enumerate}[(i)]
\item $\omega_\mu>0$ in $[-1-\sigma,1+\sigma]$,
\item $\omega_\mu$ is symmetric with respect to $0$,
\item $y\omega_\mu'(y)<0$ for all $y\not=0$.
\end{enumerate}
\end{lemma}
\begin{proof}
Fix $\mu\in(0,\mu_\si)$ and let $\omega$ be the solution of
\[
\begin{cases}
-\omega''-\lambda f'(u_\si(y))\omega=\mu\omega\quad\text{in }(-1-\sigma,1+\sigma),\\
\omega(\pm(1+\sigma))=1.
\end{cases}
\]
Since $\mu<\mu_\si$, by the maximum principle we know that $\omega>0$ in $(-1-\sigma,1+\sigma)$. Taking into account the symmetry of $u_\si$ and the maximum principle we get that $\omega(y)=\omega(-y)$  and then $(ii)$ follows.

 Moreover, from $f'\ge0$ we deduce $\omega''<0$ in $[-1-\sigma,1+\sigma]$ and then $0$ turns out to be a maximum point. The strictly concavity of $\omega$ tells also that $\omega'(y)<0$ for $y>0$ and, together to the symmetry of the function, this yields $(iii)$.
To conclude the proof set $\omega_\mu=\omega/\omega(0)$.
\end{proof}
\subsection{Construction of the domain $\Om_\e$}\label{ss1}
Now, for some $n=n(k)\in\N$, let $\mu_1,\dots,\mu_n\in\R$ be such that
\begin{equation}
\label{scelta:mu}
\frac{\mu_0}{4}>\mu_1>\dots>\mu_n>0,
\end{equation}
and for $i=1,\dots,n$
\[
\omega_i(y)=\omega_{\mu_i}(y),\quad y\in(-1-\sigma,1+\sigma),
\]
the function given by Lemma~\ref{lemma:costruzione:omegai}.  From now on, we consider $\sigma$ fixed.

Given $(t,y)\in\R\times\left(-1-\sigma ,1+\sigma \right)$, we define
\[
\tilde\phi(t,y)=\sum_{i=1}^n\alpha_i\cosh(\sqrt{\mu_i}t)\omega_i(y),
\]
for some $\alpha_i\in\R$ which will be fixed later. A straightforward computation shows that $\tilde\phi$ is a solution of the linearized problem
\[
-\Delta v=\lambda f'(u_\si(y))v,\quad\text{in }\R\times\left(-1-\sigma ,1+\sigma \right).
\]
We set $\alpha_1=-1$ while we choose $\alpha_2,\dots,\alpha_n$ in such a way that the function $\tilde\phi(t,0)=\sum_{i=1}^n\alpha_i\cosh(\sqrt{\mu_i}t)$ has $k$ nondegenerate maximum points $t_1,\dots,t_k$.
We point out that it is always possible to do this, see Lemma~\ref{lemmacosh} in the Appendix for the details. Finally, for $(x_1,\dots,x_N,y)\in\R^N\times\left(-1-\sigma ,1+\sigma \right)$ we define
\begin{equation}\label{phi}
\boxed{\phi(x_1,\dots,x_N,y)=\sum_{j=1}^N\tilde\phi(x_j,y)=\sum_{j=1}^N\sum_{i=1}^n
\alpha_i\cosh(\sqrt{\mu_i}x_j)\omega_i(y)}
\end{equation}
which solves
\[
-\Delta v=\lambda f'(u_\si(y))v,\quad\text{in }\R^N\times\left(-1-\sigma ,1+\sigma \right).
\]

We point out that, for $\e$ small enough,
\[
u_0(0)+\e\phi(0,\dots,0,0)>0,
\]
and we denote by 
\begin{equation}
\label{def:Omegae}
\boxed{\Omega_\e\hbox{ the connected component of $\set{u_\si+\e\phi>0}$ containing the origin.}}
\end{equation}
The following lemma proves some properties of the set $\Omega_\e$. The proof follows~\cite{gg3}.

\begin{lemma}
\label{lemmaomegae}
The set $\Omega_\e$ satisfies the following properties.
\begin{enumerate}[(i)]
\item $\Omega_\e\subseteq R_\e$ for $\e$ small enough, with $R_\e=[-M_\e,M_\e]^N\times[-1-\eta,1+\eta]$ where
\[
M_\e=\frac{1}{\sqrt{\mu_1}}\log\left(\frac{3\norm{u_0}_{L^\infty(-1-\eta,1+\eta)}}{\e \omega_1(1+\eta)}\right),
\]
and $\eta\in(0,\sigma)$ as small as we want.
\item $\Omega_\e\supseteq[t_1,t_{k}]^N\times\{0\}$.
\item Let $(x^\e,y^\e)\in\partial\Omega_\e$ for $\e$ small enough. Then, if $\abs{x^\e}\le C$ we have
\begin{equation}
\label{lemma2:eq1}
y^\e=\pm1+o(1),
\end{equation}
and if $\abs{x^\e}\to+\infty$ we have
\begin{equation}
\label{lemma2:eq2}
\sum_{j=1}^N\cosh(\sqrt{\mu_1}x_j^\e)=\frac{u_0(y^\e)}{\e\omega_1(y^\e)}(1+o(1)).
\end{equation}
In particular $\Omega_\e$ locally converges to the strip $\mathcal S=\R^N\times(-1,1)$ for $\e\to0$.
\item $\Omega_\e$ is symmetric with respect to the hyperplanes $x_j=0$ for $j=1,\dots,N$ and $y=0$. Moreover, it is a smooth and star-shaped domain with respect to the origin for $\e$ small enough.
\end{enumerate}
\end{lemma}
\begin{proof}
In order to prove $(i)$ we show that $u_0+\e\phi<0$ on $\partial R_\e$. First let us consider the case where $x=(x_1,\dots,x_N)\in[-M_\e,M_\e]^N$ is such that $x_j=\pm M_\e$ for some $j=1,\dots,N$ and $y\in[-1-\eta,1+\eta]$. Hence, recalling~\ref{phi}, one has
\begin{align*}
u_0(y)+\e\phi(x,y)&\le\norm{u_0}_{L^\infty(-1-\eta,1+\eta)}-\e\frac{3\norm{u_0}_{L^\infty(-1-\eta,1+\eta)}}{\e}(1+o(1))\\
&\le-\norm{u_0}_{L^\infty(-1-\eta,1+\eta)}<0
\end{align*}
as $\e\to0$.\\
Next let $(x,y)\in\set{(x,y)\in\R^{N+1}|x\in[-M_\e,M_\e]^N,\,\,y=\pm(1+\eta)}$ and observe that since $\omega_i>0$ for $y\in[-1-\eta,1+\eta]$ for all $i=1,\dots,n$ and $\alpha_1=-1$ we get
\[
\sup_{\R^N\times[-1-\eta,1+\eta]}\phi=C\in\R.
\]
Finally, we have
\[
u(x,y)\le u_0(\pm(1+\eta))+C\e<\frac{u_0(\pm(1+\eta))}{2}<0,
\]
for $\e$ small enough. Then $(i)$ follows.

Concerning $(ii)$, if $\e$ satisfies
\[
\e N\max_{t\in[t_1,t_{k}]}\left[\sum_{i=1}^n\alpha_i\cosh(\sqrt{\mu_i}t)\right]^-<\frac{u_0(0)}{2},
\]
where $[\,\,\cdot\,\,]^-$ denotes the negative part, then we get
\[
u_0+\e\phi\ge u_0-\e\phi^->\frac{u_0(0)}{2},
\]
and so $[t_1,t_{k}]^N\times\{0\}\subseteq\Omega_\e$.

To prove $(iii)$ note that from $u(x^\e,y^\e)=0$ on $\partial\Omega_\e$ we have
 \[
 u_0(y^\e)=-\e\phi(x^\e,y^\e).
 \]
If $\abs{x^\e}\le C$ then $\phi$ is uniformly bounded with respect to $\e\to0$ and then $u_0(y^\e)\to0$ which yields to $y^\e\to\pm1$.\\
On the other hand, if $\abs{x^\e}\to+\infty$ we have (recall that $\alpha_1=-1$)
\begin{equation}
\label{eq:y}
-u_0(y^\e)=-\e\sum_{j=1}^N\cosh(\sqrt{\mu_1}x_j^\e)\omega_1(y^\e)(1+o(1)),
\end{equation}
which gives~\ref{lemma2:eq2}. Moreover, since the right hand side of equation~\ref{eq:y} is strictly negative we get $u_0(y^\e)>0$ that implies $\abs{y^\e}\le1$.

The symmetry properties of the domain immediately follow from the ones of $\phi$ and $u_\si$. To show the star-shapeness with respect to the origin, it is enough to prove that there exists $\alpha>0$ such that
\[
y\partial_yu_\si(y)+\e \sum_{j=1}^Nx_j\partial_{x_j}\phi(x,y)+\e y\partial_y\phi(x,y)\le-\alpha<0,\quad\text{for all }(x,y)\in\partial\Omega_\e.
\]
Since $u_\si$ solves~\ref{PB2} we have that
\[
y\partial_yu_\si(y)<0,\quad\text{in }R_\e\setminus\set{y=0}.
\]
If $(x^\e,y^\e)\in\partial\Omega_\e$ is such that $\abs{x^\e}\le C$ as $\e\to0$, then from~\ref{lemma2:eq1} follows
\[
y^\e\partial_{y}u_\si(y^\e)=\pm\partial_{y}u_{0}(\pm1)(1+o(1))=\partial_{y}u_{0}(1)(1+o(1))<0.
\]
In this case, since the derivatives of $\phi$ are uniformly bounded with respect to $\e$, it easily follows
\begin{align*}
y\partial_{y}u_{0}(y^\e)+\e  \sum_{j=1}^Nx^\e_j\partial_{x_j}\phi(x^\e,y^\e)+\e y\partial_{y}\phi(x^\e,y^\e)&=\partial_{y}u_{0}(1)(1+o(1))+O(\e)\\
	&\le\frac{1}{2}\partial_{y}u_{0}(1)<0,
\end{align*}
for $\e$ small enough.\\
On the other hand, if $\abs{x^\e}\to+\infty$, let $\{j_1,\dots,j_m\}\subseteq\{1,\dots,N\}$ be such that $\abs{x_j}\to+\infty$ if and only if $j=j_h$ for some $h=1,\dots,m$. Then one gets
\begin{align}
&y^\e\partial_{y}u_{0}(y^\e)+\e  \sum_{j=1}^Nx^\e_j\partial_{x_j}\phi(x^\e,y^\e)+\e y^\e\partial_{y}\phi(x^\e,y^\e)\nonumber\\
&=\!y^\e\partial_{y}u_{0}(y^\e)\!+\!\e\!\sum_{j=1}^N\sum_{i=1}^n\!\alpha_i\!\left(\sqrt{\mu_i}x_j^\e\sinh(\sqrt{\mu_i}x_j^\e)\omega_i(y^\e)\!+\!\cosh(\sqrt{\mu_i}x_j^\e)y^\e\partial_{y}\omega_i(y^\e)\right)\nonumber\\
&\label{b0}\le y^\e\partial_{y}u_{0}(y^\e)-\frac{\e}{2}\sum_{h=1}^m\sqrt{\mu_1}x_{j_h}^\e\sinh(\sqrt{\mu_1}x_{j_h}^\e)\omega_1(y^\e)\big((1+o(1)\big).
\end{align}
For $h=1,\dots,m$ we have that $-x_{j_h}\sinh(\sqrt{\mu_1}x_{j_h})\le-\cosh(\sqrt{\mu_1}x_{j_h})$ and then
\begin{align*}
-\sum_{h=1}^mx_{j_h}^\e\sinh(\sqrt{\mu_1}x_{j_h}^\e)\big((1+o(1)\big)&\le-\sum_{h=1}^m\cosh(\sqrt{\mu_1}x_{j_h}^\e)\big((1+o(1)\big)\\\
&=-\sum_{j=1}^N\cosh(\sqrt{\mu_1}x_j^\e)\big((1+o(1)\big).
\end{align*}
So we have that~\ref{b0} becomes
\begin{align*}
y^\e\partial_{y}u_{0}(y^\e)&+\e\sum_{j=1}^Nx^\e_j\partial_{x_j}\phi(x^\e,y^\e)+\e y^\e\partial_{y}\phi(x^\e,y^\e)\\
&\le\,\, y^\e \partial_{y}u_{0}(y^\e)-\frac{\e}{2}\sqrt{\mu_1}\omega_1(y^\e)\sum_{j=1}^N\cosh(\sqrt{\mu_1}x_j^\e)(1+o(1))\\
        &\!\!\!\!\overset{\ref{lemma2:eq2}}{\le}\!y^\e\partial_{y}u_{0}(y^\e)-\frac{\sqrt{\mu_1}}{2}u_0(y^\e)(1+o(1))\\
&\le\,\,  y^\e\partial_{y}u_{0}(y^\e)-\frac{\sqrt{\mu_1}}{4}u_0(y^\e),
\end{align*}
and if $y^\e\partial_{y}u_{0}(y^\e)-\frac{\sqrt{\mu_1}}{4}u_0(y^\e)\to0$, since both terms are nonpositive, then they both go to $0$. This implies $y^\e\to0$ in the first term, and $y^\e\to1$ in the second one, a contradiction.

 Hence $y^\e\partial_{y}u_{0}(y^\e)-\frac{\sqrt{\mu_1}}{4}u_0(y^\e)\le-\tilde\alpha$.
Finally, for
\[
\alpha=\min\left\{-\frac{1}{2}\partial_{y}u_{0}(1),\tilde\alpha\right\},
\]
we have the claim.

Of course $y\partial_{y}u_0(y)+\e\sum_{j=1}^Nx_j\partial_{x_j}\phi(x,y)+\e y\partial_{y}\phi(x,y)\not=0$ on $\partial\Omega_\e$ implies that $\partial\Omega_\e$ is a smooth set.
\end{proof}

Next lemma tell us that the function $u_0+\e\phi$ has many critical points.
\begin{lemma}
\label{propmassimi}
The function $u_0+\e\phi$ has at least $k$ different nondegenerate local maxima in $\Omega_\e$ for $\e$ small enough.
\end{lemma}
\begin{proof}
Set $U=u_0+\e\phi$ and let $t_1<\dots<t_k$ be local, nondegenerate maxima for $\tilde\phi(t,0)=\sum_{i=1}^n\alpha_i\cosh(\sqrt{\mu_i}t)$. Then a straightforward computation gives
\[
\nabla U(t_m,\dots,t_m,0)=0.
\]
Next observing that $\partial_{yy}u_0(0)=-\lambda f(u_0(0))<0$ we have
\begin{align}
\partial_{yy}U(t_m,\dots,t_m,0)&=\partial_{yy}u_0(0)+\e\sum_{j=1}^N\sum_{i=1}^k\alpha_i\cosh(\sqrt{\mu_i}t_m)\partial_{yy}\omega_i(0)\nonumber\\
\label{hessiana}	&<-\frac{\lambda}{2} f(u_0(0))<0,
\end{align}
for $\eps$ small enough and for all $m=1,\dots,k$. Finally in $(t_m,\dots,t_m,0)$ one has
\begin{align*}
\partial_{x_jx_j}U&=\e\sum_{i=1}^k\alpha_i{\mu_i}\cosh(\sqrt{\mu_i}t_m)<0,\\
\partial_{x_\ell x_j}U&=0,\qquad\forall\ell\not=j,\\
\partial_{x_jy}U&=\e\sum_{i=1}^k\alpha_i\sqrt{\mu_i}\sinh(\sqrt{\mu_i}t_m)\partial_y\omega_i(0)=0,
\end{align*}
which, together to~\ref{hessiana} show us that the Hessian matrix of $U$ is negative definite in $(t_m,\dots,t_m,0)$  for all $m=1,\dots,k$ and the proof is complete.
\end{proof}
Now we prove that problem~\ref{PB0} admits a stable solution in the domain $\Omega_\e$ for many $\lambda's$.
\begin{lemma}
\label{lemmalambda*}
For $\e$ small enough, it holds
\[
\lambda^*(\Omega_\e)\ge\la^*(-1,1).
\]
\end{lemma}
\begin{proof}
Let us write $\lambda^*=\lambda^*(-1,1)$ for simplicity. For $\eta>0$ small enough we have
\[
\lambda_\eta^*=\lambda^*(-1-\eta,1+\eta)=\frac{\lambda^*}{(1+\eta)^2}>\lambda,
\]
and by $u_\eta^*$ the solution of
\[
\begin{cases}
-u''=\lambda_\eta^*f(u)&\hbox{in }(-1-\eta,1+\eta)\\
u>0&\hbox{in }(-1-\eta,1+\eta)\\
u(\pm(1+\eta))=0.
\end{cases}
\]
Now, let $\e$ so small that $\Omega_\e\subseteq\R^N\times(-1-\eta,1-\eta)$, then $u_\eta^*$ is a \emph{supersolution} of problem
\[
\begin{cases}
-u''=\lambda_\eta^{*}f(u)&\hbox{in }\Omega_\e\\
u>0&\hbox{in }\Omega_\e\\
u=0&\hbox{on }\partial\Omega_\e
\end{cases}
\]
 that is $-\Delta u_\eta^*\ge\lambda_\eta^*f(u_\eta^*)$ in $\Omega_\e$ and $u_\eta^*\ge0$ on $\partial\Omega_\e$ (here we follows the notations in~\cite{Bandle(book)}). Then~\cite[Theorem 4.7]{Bandle(book)} ensures that $\lambda^*(\Omega_\e)\ge\lambda_\eta^*>\lambda$.
\end{proof}
Finally, for $\e>0$, we define  
\begin{equation}
\label{ue}
\boxed{u_\e\hbox{ as a stable solution of problem~\ref{PB0} in } \Omega_\e.}
\end{equation}
\subsection{Properties of the function $u_\e$}
Before to state the main properties of the solution $u_\e$ we compute the eigenvalues of a related operator. The proof uses the classical separation of variables.
\begin{lemma}
\label{lemma1}
Denote by $\mu_{1,\si}(R)$ the first eigenvalue of the operator $-\Delta-\lambda f'(u_\si(y))$ in the rectangle
\[
R=\prod_{j}^N(a_j,b_j)\times(-1-\sigma ,1+\sigma ),
\]
with $u_{|\partial R}=0$, where $a_j<b_j$ for all $j=1,\dots,N$. Then
\[
\mu_{1,\si}(R)=\mu_\si+\sum_{j=1}^N\left(\frac\pi{b_j-a_j}\right)^2>\mu_\si.
\]
\end{lemma}
\begin{proof}
Fix $\mu\in\R$ and let $A_j$ and $B$ be positive solutions of
\begin{equation}
\label{pippo}
\begin{cases}
A_j''(t)=c_j A_j(t)\quad&\text{in }(a_j,b_j)\\
A_j(a_j)=A_j(b_j)=0
\end{cases}
\end{equation}
and
\begin{equation}
\label{pluto}
\begin{cases}
-B''(y)-\left(\lambda f'\left(u_\si(y)\right)+\mu\right)B(y)=\sum_{j=1}^Nc_j B(y)\quad\text{in }(-1-\sigma ,1+\sigma )\\
B(\pm(1+\sigma))=0
\end{cases}
\end{equation}
for some $c_j\in\R$. We have that the solution of~\ref{pippo} is given by
\[
A_j(t)=\alpha\sin\left(\sqrt{-c_j}(t-a_j)\right)
\]
with $\alpha\in\R$ and
\[
c_j=-\left(\frac\pi{b_j-a_j}\right)^2<0
\]
and from~\ref{pluto} it follows
\[
\sum_{j=1}^Nc_j+\mu=\mu_\si.
\]
Finally, since
\[
v(x,y)=B(y)\prod_{j}^NA_j(x_j),
\]
solves
\[
\begin{cases}
-\Delta v-\lambda f'(u_\si(y))v=\mu v&\hbox{in }R\\
v=0&\hbox{on }\partial R
\end{cases}
\]
and $v>0$ we conclude that
\[
\mu_{1,\si}(R)=\mu=\mu_\si-\sum_{j=1}^Nc_j=\mu_\si+\sum_{j=1}^N\left(\frac\pi{b_j-a_j}\right)^2>\mu_\si.\qedhere
\]
\end{proof}
\begin{remark}
\label{sez2:rmk1}
From $(i)$ of Lemma~\ref{lemmaomegae} and the previous lemma, one has that the first eigenvalue of the operator $-\Delta-\lambda f'(u_\si(y))$ with Dirichlet boundary conditions in $\Omega_\e$ is strictly positive.
\end{remark}
The rest of the section is devoted to show that the solution $u_\e$ defined in~\ref{ue} is close to $u_0+\e\phi$ as $\e\to0$. By Lemma \ref{propmassimi} then $(iv)$ of Theorem~\ref{THM1} follows.

Let us start with the following bound for $u_\e$.

\begin{lemma}
\label{lemmaconvergenza1bis}
There exists a function $h:(0,+\infty)\to(0,+\infty)$ such that $h(\e)\to0$ for $\e\to0$ and $u_\e-u_0\le h(\e)$ in $\Omega_\e$ uniformly with respect to $(x,y)\in\Omega_\e$.
\end{lemma}
\begin{proof}
For $\eta>0$, let $u_	\eta$ be the  stable solution of
\[
\begin{cases}
-u''=\lambda f(u)&\hbox{in }(-1-\eta,1+\eta)\\
u>0&\hbox{in }(-1-\eta,1+\eta)\\
u(\pm(1+\eta))=0.
\end{cases}
\]
For $\e$ small enough such that $\Omega_\e\subseteq\R^N\times(-1-\eta,1+\eta)$, from the convexity of $f$ we have
\[
\begin{cases}
-\Delta(u_\e-u_\eta)=\lambda \left(f(u_\e)-f(u_\eta)\right)\le \lambda f'(u_\e)(u_\e-u_\eta)& \hbox{in}\ \Omega_\eps\\
u_\e-u_\eta<0& \hbox{on}\ \partial\Omega_\e
\end{cases}
\]
and then from the stability of $u_\e$ we can apply the maximum principle to deduce $u_\e\le u_\eta$ in $\Omega_\e$. For $(x,y)\in\Omega_\e$, by the maximum principle applied to $u_\eta-u_0$ we get
\[
u_\e(x,y)-u_0(y)\le u_\eta(y)-u_0(y)\le\max(u_\eta-u_0)_{|y=\pm(1+\eta)}=-u_0(1+\eta).
\]
Next let us define the function $h(\e)$ as follows: for any $\e>0$ let $\eta(\e)$ be the smallest positive number such that $\Omega_\e\subseteq\R^N\times(-1-\eta(\e),1+\eta(\e))$. By the properties of $\Omega_\e$ we have that $\eta(\e)\to0$ as $\e\to0$. Finally, as $\e\to0$
\[
h(\e)=-u_0\big(1+\eta(\e)\big)\to0,
\]
which gives the claim.
\end{proof}
Next Lemma gives a first approximation of the closeness of $u_\e$ to $u_0+\e\phi$. It will be improved later.
\begin{lemma}
\label{lemmaconvergenza2bis}
Given $\psi_\e=\frac{u_\eps-u_\si-\eps\phi}{\e}$ one has $0\le\psi_\e<\bar\psi$ in $\Omega_\e$ for $\e$ small enough, where
\[
\bar\psi(x,y)=\sum_{j=1}^N\sum_{i=1}^n\abs{\alpha_i}\left(\omega_i(y)-C_{i}\right)\cosh(\sqrt{\mu_i}x_j),
\]
with $0<C_i<\inf\limits_{(-1-\eta,1+\eta)}\omega_i$ for all $i=1,\dots,k$ and $0<\eta<\sigma$ small, fixed.
\end{lemma}
\begin{proof}
Using the convexity of $f$ we have
\[
-\Delta\psi_\e-\lambda f'(u_\si)\psi_\e\ge0.
\]
Moreover, $\psi_\e=0$ on $\partial\Omega_\e$ and taking into account Remark~\ref{sez2:rmk1} we can apply the maximum principle to get $\psi_\e>0$ in $\Omega_\e$.

Again from the convexity of $f$ we have
\begin{align}
-\Delta\psi_\e-\lambda f'(u_\e)\psi_\e&\le\lambda \left(f'(u_\e)-f'(u_\si)\right)\phi\nonumber\\
\label{lemmaconvergenza2bis:eq1}
&=\lambda \sum_{j=1}^N\sum_{i=1}^n\alpha_i\left(f'(u_\e)-f'(u_\si)\right)\cosh(\sqrt{\mu_i}x_j)\omega_i(y).
\end{align}
From the  definition of $C_i$ it holds $\bar\psi>0$ on $\overline\Omega_\e$.
Furthermore, in $\Omega_\eps$ we have that $\bar\psi$ verifies
\[
-\Delta \bar\psi=\sum_{j=1}^N\sum_{i=1}^n\abs{\alpha_i}\left(\lambda f'(u_\si)\omega_i(y)+\mu_iC_i\right)\cosh(\sqrt{\mu_i}x_j),
\]
and then
\begin{align}
-\Delta &\bar\psi-\lambda f'(u_\e)\bar\psi\nonumber\\
\label{lemmaconvergenza2bis:eq2}
&=\sum_{j=1}^N\sum_{i=1}^n\abs{\alpha_i}\left[\lambda \left(f'(u_\si)-f'(u_\e)\right)\omega_i(y)+(\lambda f'(u_\e)+\mu_i)C_i\right]\cosh(\sqrt{\mu_i}x_j).
\end{align}
Moreover
\[
f'(u_\e)-f'(u_\si)=f''\left(t_\e u_\e+(1-t_\e)u_\si\right)(u_\e-u_\si),
\]
with $t_\e=t_\e(x,y)\in(0,1)$ .\\
From Lemma~\ref{lemmaconvergenza1bis} we have $u_\e-u_0\le h(\e)$ with $h>0$ and $h\to0$ as $\e\to0$. Since $f''$ is positive and $t_\e u_\e+(1-t_\e)u_\si$ is bounded uniformly with respect to $\e$ we get
\[
\lambda\left(f'(u_\e)-f'(u_0)\right)\le Ch(\e),
\]
for some $C>0$.
Finally from~\ref{lemmaconvergenza2bis:eq1} and~\ref{lemmaconvergenza2bis:eq2} we deduce that
\begin{align*}
-\Delta&(\psi_\e-\bar\psi)-\lambda f'(u_\e)(\psi_\e-\bar\psi)\\
&\le\sum_{j=1}^N\sum_{i=1}^n\left[(\abs{\alpha_i}+\alpha_i)\lambda (f'(u_\e)-f'(u_0))\omega_i(y)-\abs{\alpha_i}(\lambda f'(u_\e)+\mu_i)C_i\right]\cosh(\sqrt{\mu_i}x_j)\\
&\le\sum_{j=1}^N\sum_{i=1}^n\big[(\abs{\alpha_i}+\alpha_i)Ch(\e)-\underbrace{\abs{\alpha_i}(\lambda f'(u_\e)+\mu_i)C_i}_{\le-\abs{\alpha_i}\mu_iC_i}\big]\cosh(\sqrt{\mu_i}x_j)\le0,
\end{align*}
for $\e$ small enough, which gives
\[
\begin{cases}
-\Delta(\psi_\e-\bar\psi)-\lambda f'(u_\e)(\psi_\e-\bar\psi)\le0 &\text{in}\ \Omega_\eps\\
\psi_\e-\bar\psi<0&\text{on}\ \partial\Omega_\e
\end{cases}
\]
and the maximum principle provides $\psi_\e-\bar\psi<0$ in $\Omega_\e$.
\end{proof}
Next lemma gives us the final estimate. Here it will be crucial to choose the coefficients $\mu_i$ as in~\ref{scelta:mu}.
\begin{lemma}
\label{lemmaconvo(e)}
Let
\[
\Psi_\e=\frac{u_\eps-u_0-\eps\phi}{\e^2}.
\]
Then in every $K\subset\!\subset\Omega_\e$ one has $\abs{\Psi_\e}\le C$, for some $C=C(K)>0$ and $\e$ small enough.
\end{lemma}
\begin{proof}
Let us denote by $C$ any positive constant which does not depend on $\e$.
Consider the function $F(\e)=f(u_0+\e\phi+\e^2\Psi_\e)$. Then for $\e$ small there exists $t_\e=t_\e(x,y)\in(0,1)$ such that
\begin{align}\label{lemma:conv:2:eq:1}	
f(u_\e)=F(\e)=f(u_0)&+\e f'(u_0)\phi+\frac{\e^2}{2}f''(u_0)\phi^2+\e^2f'(u_0)\Psi_\e+\nonumber\\
	&+\frac{\e^3}{6}f'''(u_0+t_\e\e\phi+t_\e^2\e^2\Psi_\e)(\phi+2t_\e\e\Psi_\e)^2+\nonumber\\
	&+\e^3f''(u_0+t_\e\e\phi+t_\e^2\e^2\Psi_\e)(\phi+2t_\e\e\Psi_\e)\Psi_\e.
\end{align}
From the previous lemma we have that $0\le\e\Psi_\e\le\bar\psi\le C\sum_{j=1}^N\cosh(\sqrt{\mu_1}x_j)$. From Lemma~\ref{lemmaomegae}, $\abs{x_j}\le C\log(1/\e)$ for all $j=1,\dots,N$ and then
\[
\left|u_0+t_\e\e\phi+t_\e^2\e^2\Psi_\e\right|\le C,\quad\text{in }\Omega_\e.
\]
In $\Omega_\e$, taking into account~\ref{lemma:conv:2:eq:1}, we have the following inequality
\begin{align*}
f(u_\e)-f(u_0)-\e f'(u_0)\phi&\le C\e^2\left(\phi^2+\e(\phi+2\bar\psi)^2+(\phi+2\bar\psi)\bar\psi\right)+\e^2f'(u_0)\Psi_\e\\
	&\le \frac{C_\infty}{\lambda}\e^2\sum_{j=1}^N\cosh(2\sqrt{\mu_1}x_j)+\e^2f'(u_0)\Psi_\e,
\end{align*}
for some $C_\infty>0$, that implies
\begin{equation}
\label{eq:PSI}
-\Delta\Psi_\e-\lambda f'(u_0)\Psi_\e\le C_\infty\sum_{j=1}^N\cosh(2\sqrt{\mu_1}x_j).
\end{equation}
Fix $\mu_\infty=4\mu_1$. Note that $\mu_\infty<\mu_0$ thanks to~\ref{scelta:mu}. Then taking into account Lemma~\ref{lemma:costruzione:omegai} set $\omega_\infty=\omega_{\mu_\infty}$ and for $(x,y)\in\R^N\times(1-\sigma,1+\sigma)$ consider
\[
\psi_\infty(x,y)=\frac{C_\infty}{c_\infty\mu_\infty}\sum_{j=1}^N\left(\omega_\infty(y)-c_\infty\right)\cosh(\sqrt{\mu_\infty}x_j),
\]
where  $0<c_\infty<\inf\limits_{(-1-\sigma,1+\sigma)}\omega_\infty$.

Clearly $\psi_\infty>0$ in $\overline\Omega_\e$ and $\psi_\infty$ satisfies the following inequality
\begin{align*}
-\Delta\psi_\infty-\lambda f'(u_0)\psi_\infty&=\frac{C_\infty}{c_\infty\mu_\infty}\sum_{j=1}^Nc_\infty\left(\mu_\infty+\lambda f'(u_0)\right)\cosh(\sqrt{\mu_\infty}x_j)\\
	&\ge C_\infty\sum_{j=1}^N\cosh(2\sqrt{\mu_1}x_j),
\end{align*}
which together to~\ref{eq:PSI} gives
\[
\begin{cases}
-\Delta(\Psi_\e-\psi_\infty)-\lambda f'(u_0)(\Psi_\e-\psi_\infty)\le0 &\text{in}\ \Omega_\eps\\
\Psi_\e-\psi_\infty<0&\text{on}\ \partial\Omega_\e
\end{cases}
\]
and again the maximum principle provides $\Psi_\e-\psi_\infty<0$ in $\Omega_\e$. For $C(K)=\max_K \psi_\infty$ the proof is complete.
\end{proof}

\subsection{Proof of Theorem~\ref{THM1}}
\begin{proof}

We have that $(i)$ and $(ii)$ follow by $(iv)$ and $(iii)$ of Lemma~\ref{lemmaomegae} respectively.
The proof of $(iii)$ is given in Lemma~\ref{lemmalambda*}.
 
Let us prove $(iv)$.
By Lemma~\ref{propmassimi} we have that $u_0+\e\phi$ admits $k$ strict maxima points. Fix a compact set $K\subset\!\subset\Omega_\e$ containing such points.
On the other hand Lemma~\ref{lemmaconvo(e)} implies $u_\e=u_0+\e\phi+O(\e^2)$ in $K$ and so the claim follows.
\end{proof}
\begin{remark}
\label{rmk:general}
We can prove a little more general version of Theorem~\ref{THM1}: indeed assumption~\ref{i3} can be dropped and we can simply ask that there exists $u_0$ stable solution of
\[
\begin{cases}
-u''=g(u)&\hbox{in }(-1,1)\\
u>0&\hbox{in }(-1,1)\\
u(\pm1)=0.
\end{cases}
\]
Finally we build $\Omega_\e$ as before and then ask for the existence of a stable solution $u_\e$ of problem~\ref{PB1} in $\Omega_\e$.
\end{remark}

\begin{remark}\label{R}
Let us show that the assumption that $u_\e$ is a \emph{stable} solution is crucial in our construction. To do this 
let us assume $N=1$ for simplicity and consider $f(t)=\lambda_1t$, where $\lambda_1$ is the first eigenvalue of the Dirichlet problem. In this case the first eigenvalue of the linearized problem at the first eigenfunction is $0$. Let us see that it is not possible to construct a domain $\Omega_\e$ as in the previous section.
Indeed if we argue as before we have that $u_0(y)=\cos\left(\frac{\pi}{2}y\right)$ is the solution of
\[
\begin{cases}
-u''=\frac{\pi^2}4u&\hbox{in }(-1,1)\\
u>0&\hbox{in }(-1,1)\\
u(\pm1)=0.
\end{cases}
\]
Now, for $n\in\N$, $\alpha_i\in\R$ (again with $\alpha_1=-1$) and $\mu_i>0$ for $i=1,\dots,n$, we have that
\[
\phi(x,y)=\sum_{i=1}^n\alpha_i\cosh(\sqrt{\mu_i}x)\cos\left(\sqrt{\pi^2/4+\mu_i}y\right),
\]
solves the linearized problem, i.e.
\[
-\Delta\phi=\frac{\pi^2}{4}\phi\quad\hbox{in }\R^2,
\]
As for the general case we observe that $u_0(0)+\e\phi(0,0)>0$ for $\e$ small enough and then we set $\Omega_\e=\set{u_0+\e\phi>0}$. Now for any $\mu_1>0$ set
\[
\bar y=\frac{\frac\pi2}{\sqrt{\pi^2/4+\mu_1}}\in(0,1),
\]
and then we can find $\delta>0$ sufficiently small such that if $\e$ is small enough it holds
\[
\R\times\set{y=\bar y+\delta}\subseteq\Omega_\e,
\]
showing that the domain $\Omega_\e$ is not bounded. This shows that our construction fails.
\end{remark}

\section{The torsion problem: proof of Theorem~\ref{THM2}}
\label{SEZ4}
In this section we take $x\in\R$ and $y=(y_1,\dots,y_N)\in\R^N$ and we assume the hypothesis of Theorem~\ref{THM2}. We construct a solution $u_\e$  of the torsion problem ($g(u)=\mathrm{Const.}$) with $k$ maximum points in a domain $\Omega_\e$ whose boundary has {\em positive mean curvature}. Here
the domain $\Omega_\e$ and the function $u_\e$ are similar to the ones defined in Section~\ref{SEZ2}. 

Let us start by introducing the following function $u_\e:\R^{N+1}\to\R$, given by
\[
u_\e(x,y)=u_0(y)+\e\phi(x,y)\quad x\in\R,\ y\in\R^N,
\]
where 
\[
u_0(y)=\frac{1}{2}\sum_{j=1}^N\left(1-y_j^2\right)=\frac{1}{2}\left(N-\abs{y}^2\right),
\] 
which solves
\begin{equation}
\begin{cases}
-\Delta u=N&\hbox{in }\mathcal C\\
u=0&\hbox{on }\partial\mathcal C
\end{cases}
\end{equation}
in the cylinder $\mathcal C=\{(x,y)\in\R^{N+1}| \abs{y}^2<N\}$.  Finally $\phi$ is an harmonic function in the whole $\R^{N+1}$ defined by
\[
\phi(x,y)=\sum_{j=1}^Nv(x,y_j),
\]
where $v(t,s)=\Re (F_k(t+is))$, for $t,s\in\R$ with 
\begin{align*}
F_k(t+is)&=-\prod_{\ell=1}^{k}\left[(t-t_\ell+is)(t+t_\ell+is)\right]\\
	&=-\prod_{\ell=1}^{k}\left(t^2-s^2-t_\ell^2+2its\right),\qquad\quad\text{for } 0<t_1<\dots<t_{k},
\end{align*}
and $\Re(\cdot)$ stands for the real part of a complex function. 
Note that $v$ is symmetric with respect to both $\{t=0\}$ and $\{s=0\}$ and it can be written as
\begin{equation}
\label{prv}
v(t,s)=-\sum_{h=0}^{2k}a_hP_h(t,s),
\end{equation}
where $P_h$ is an harmonic polynomial of degree $h$, $a_{2k}=1$ and 
\begin{equation}
\label{prv2}
P_{2k}(t,s)=\sum_{\ell=0}^kb_\ell t^{2k-2\ell}s^{2\ell},\quad b_0=b_k=1.
\end{equation}
Resuming we have that for $x\in\R$ and $y\in\R^N$
\[
\boxed{\begin{split}u_\e(x,y)&=u_0(y)+\e\phi(x,y)\\
&=\frac{1}{2}\left(N-\abs{y}^2\right)+\e\sum_{j=1}^Nv(x,y_j)\\
&=\frac{1}{2}\sum_{j=1}^N\left(1-y_j^2\right)-\e\sum_{j=1}^N\sum_{h=0}^{2k}a_hP_h(x,y_j).\end{split}}
\]
Since $F_k:\C\to\C$ is holomorphic, it easily follows that $\phi$ is harmonic and then $u_\e$ satisfies $-\Delta u_\e=N$.
Finally, we point out that $\partial_{y_iy_j}u_\eps=0$ for all $i\not=j$.

\subsection{Preliminary results}

In this section we show some properties of the function $u_\e$ and of the domain $\Omega_\e$ that we are going to define.

As in Section~\ref{SEZ2} we point out that
\[
u_\eps(0,0,\dots,0)=\frac{N}{2}+\eps\sum_{j=1}^Nv(0,0)\ge\frac{N}{4}>0,
\]
for $\e$ small enough and we denote by $\Omega_\e$ the connected component of $\set{u_0+\e\phi>0}$ containing the origin.

The following lemma proves some properties of the set $\Omega_\e$. 
\begin{lemma}
\label{lemma:torsione1}
The set $\Omega_\e$ satisfies the following properties.
\begin{enumerate}[(i)]
\item $\Omega_\e\subseteq C_\e$ for $\e$ small enough, where
\[
C_{\eps}=\Set{(x,y)\in\R^{N+1}|x\in(-M_\eps,M_\eps),\, \abs{y}^2< N(1+\eta)^2},
\]
for some $0<\eta<1$, and $M_\eps=\eps^{-\frac{1}{2k}}$. 
\item $\Omega_\e\supseteq[-t_{k},t_{k}]\times\{0\}^N$.
\item Let $(x^\eps,y^\eps)\in\partial\Omega_\eps$. If $\abs{y^\eps}\to0$ then we have
\begin{equation}
\label{lemma:curv2:eq1}
\abs{x^\eps}=\left(2\eps\right)^{-\frac{1}{2k}}(1+o(1))\to+\infty.
\end{equation}
On the other hand, if $\abs{x^\e}\le C$, then
\[
\abs{y^\eps}^2\to N.
\]
\item $\Omega_\e$ is symmetric with respect to the hyperplanes $x=0$ and $y_j=0$ for $j=1,\dots,N$. Moreover, it is a smooth and star-shaped domain with respect to the origin for $\e$ small enough.
\end{enumerate}
\end{lemma}
\begin{proof}
To prove $(i)$ we firstly show that
\begin{equation}
\label{eqausiliaria}
u_\eps\le-1/2,\quad\text{on }\Set{(x,y)\in\R^{N+1}|x=\pm M_\e,\,\abs{y}^2<N(1+\eta)^2},
\end{equation}
for $\e$ small enough. Indeed by~\ref{prv2} we get
\[
\eps P_{2k}(\pm M_\eps,s)=\eps\sum_{\ell=0}^kb_\ell\left(\eps^{-\frac{1}{2k}}\right)^{2k-2\ell}s^{2\ell}=1+o(1),\quad\text{as }\eps\to0,
\]
uniformly with respect to $\abs{s}<\sqrt N(1+\eta)$.
Similarly we have
\[
\eps P_{h}(\pm M_\eps,s)=o(1),\quad\text{for all }0\le h\le2k-1.
\]
Finally, for $x=\pm M_\eps$ and $\abs{y}^2\le N(1+\eta)^2$ we have
\begin{align*}
u_\eps(x,y)\le\frac{N}{2}+\eps\sum_{j=1}^Nv(\pm M_\eps,y_j)(1+o(1))=\frac{N}{2}-N+o(1)\le-\frac{1}{2}.
\end{align*}

On the other hand by~\ref{prv} and since $a_{2k}=1$ we get
\[
\sup_{t\in\R}\max_{s\in[-\sqrt N(1+\eta),\sqrt N(1+\eta)]}v(t,s)=C\in\R.
\]
Then for all $(x,y)\in\overline C_\eps$ with $\abs{y}^2=N(1+\eta)^2$ we obtain
\[
u_\eps(x,y)=-\frac{N}{2}\eta^2-N\eta+\eps\sum_{j=1}^Nv(x,y_j)<-\frac{N}{2}\eta^2<0,
\]
for $\eps$ small enough which together to~\ref{eqausiliaria} proves $(i)$.

Concerning $(ii)$, we know that the origin belongs to $\Omega_\e$ and since $u_\eps$ is continuous, then $\Omega_\eps$ is an open and connected set. Finally if $\eps$ satisfies
\[
\eps<\frac{u_0(0,\dots,0)}{\max_{x\in[-t_{k},t_{k}]}(-\phi(x,0,\dots,0))},
\]
then $[-t_k,t_{k}]\times\{0\}^N\subseteq\Omega_\eps$.

In order to prove $(iii)$, let $(x^\eps,y^\eps)\in\partial\Omega_\eps$. Then one has
\begin{equation}
\label{lemma:curv2:eq2}
\frac{1}{2}\left(N-\abs{y^\e}^2\right)=-\eps\sum_{j=1}^Nv(x^\eps,y_j^\eps).
\end{equation}
If $\abs{x^\eps}\le C$, $v(x^\eps,y_j^\eps)$ is bounded and then we easily get $\abs{y^\eps}^2\to N$.

Then we can assume $\abs{x^\eps}\to+\infty$.
In particular, for all $j=1,\dots,N$, it holds $v(x^\eps,y_j^\eps)=-(x^\eps)^{2k}(1+o(1))$ and from~\ref{lemma:curv2:eq2} we get
\[
(x^\eps)^{2k}=\frac{1}{2}\left(1-\frac{\abs{y^\eps}^2}{N}\right)\eps^{-1}(1+o(1))=\frac{1}{2}\eps^{-1}(1+o(1)),
\]
and in particular~\ref{lemma:curv2:eq1} holds.

The symmetry properties of the domain immediately follow from the ones of $u_\e$. Then to finish the proof it is enough to prove that there exists $\alpha>0$ such that
\[
x\partial_{x}u_\eps+\sum_{j=1}^Ny_j\partial_{y_j}u_\eps\le-\alpha<0,\quad\text{for all }(x,y)\in\partial\Omega_\eps.
\]
We have
\[
x\partial_{x}u_\eps+\sum_{j=1}^Ny_j\partial_{y_j}u_\eps=-\sum_{j=1}^Ny_j^2+\eps\sum_{j=1}^N\left(xv_t(x,y_j)+y_jv_{s}(x,y_j)\right).
\]
On the other hand since $u_\eps(x,y)=0$ on $\partial\Omega_\eps$ we have 
\[
\sum_{j=1}^Ny_j^2=N+2\eps\sum_{j=1}^Nv(x,y_j),
\]
and then
\[
x\partial_{x}u_\eps+\sum_{j=1}^Ny_j\partial_{y_j}u_\eps=-N+\eps\sum_{j=1}^N\big(xv_t(x,y_j)+y_jv_s(x,y_j)-2v(x,y_j)\big).
\]
Since we have that 
\begin{align*}
tv_t(t,s)+sv_s(t,s)-2v(t,s)&
=-\sum_{h=0}^{2k}a_h\left(t\partial_tP_h(t,s)+s\partial_sP_h(t,s)-2P_h(t,s)\right)\\
&=-\sum_{h=0}^{2k}(h-2)a_hP_h(t,s)\to-\infty,
\end{align*}
for $\abs{t}\to+\infty$ uniformly with respect to $\abs{s}<\sqrt N(1+\eta)$. Hence
\[
\sup_{(t,s)\in\R\times[-\sqrt N(1+\eta),\sqrt N(1+\eta)]}tv_t(t,s)+sv_s(t,s)-2v(t,s)=d<+\infty,
\]
and then
\[
\sum_{j=1}^N\left(xv_t(x,y_j)+y_jv_s(x,y_j)-2v(x,y_j)\right)\le Nd<+\infty.
\]
Finally
\[
\sup_{\partial\Omega_\eps}\left(x\partial_{x}u_\eps+\sum_{j=1}^Ny_j\partial_{y_j}u_\eps\right)\le-N+o(1)\le-\frac{N}{2},
\]
for $\eps$ small enough. Of course $x\partial_{x}u_\eps+\sum_{j=1}^Ny_j\partial_{y_j}u_\eps\not=0$ on $\partial\Omega_\eps$ implies that $\partial\Omega_\eps$ is a smooth hypersurface.
\end{proof}

\begin{remark}
\label{rmk1}
In particular from $(iii)$ of Lemma~\ref{lemma:torsione1} we deduce that  $\Omega_\e$ locally converges to the cylinder $\mathcal C=\set{(x,y)\in\R^{N+1}| \abs{y}^2<N}$.
\vskip0.2cm
Equation~\ref{lemma:curv2:eq1} will be useful in the computation of the curvature of $\partial\Omega_\e$ in next subsection.
\end{remark}

\begin{lemma}
\label{lemma:torsione2}
The function $u_\eps$ has at least $k$ different nondegenerate local maxima in $\Omega_\eps$ for $\eps$ small enough.
\end{lemma}
\begin{proof}
The proof is similar to the one of Lemma~\ref{propmassimi}.

For
\[
q(t)=\Re\left(F_k(t+i0)\right)=-\prod_{\ell=1}^{k}(t-t_\ell)(t+t_\ell)=v(t,0),
\]
we have $q(t)=0$ if and only if $t=\pm t_\ell$ for some $\ell=1,\dots,k$ and $q(t)\to-\infty$ as $\abs{t}\to+\infty$. Now assume $k$ even, the case $k$ odd follows by minor changes. Then there exist $\bar t_\ell \in(t_{2\ell +1},t_{2\ell +2})$ with $\ell =0,\dots,k/2$ such that
\[
q'(\bar t_\ell)=0,\quad\text{and}\quad q''(\bar t_\ell)<0\quad\forall\ell =0,\dots,k/2,
\]
see also Lemma~\ref{lemmanondegen}.

Moreover, from the definition of $v$, since every time a power of $s$ appears then it is an even power, we get that $\partial_{s}v(t,0)=\partial_{ts}v(t,0)=0$ for all $t\in\R$.
Then a straightforward computation gives
\[
\nabla u_\e(\bar t_\ell,0,\dots,0)=0.
\]
Next, for all $j=1,\dots,N$ and for all $\ell=0,\dots,k/2$, we have
\begin{equation}
\label{hessiana2}
\partial_{y_jy_j}u_\e(\bar t_\ell,0,\dots,0)=-1+\e \partial_{ss}v(\bar t_\ell,0)<0,
\end{equation}
for $\eps$ small enough. Finally in $(\bar t_\ell,0,\dots,0)$ one has
\begin{align*}
\partial_{xx}u_\e&=\e N q''(\bar t_\ell,0)<0,\\
\partial_{y_i y_j}u_\e&=0,\qquad\forall i\not=j,\\
\partial_{xy_j}u_\e&=\e\partial_{ts} v(\bar t_\ell,0)=0,
\end{align*}
which, together to~\ref{hessiana2}, show us that the Hessian matrix of $u_\e$ is negative definite in $(\bar t_\ell,0,\dots,0)$ for all $\ell=0,\dots,k/2$ and the proof is complete since $u_\e$ is even in the $x$ variable.
\end{proof}

\begin{remark}
We point out that $\Omega_\eps$ is not convex. Indeed, we know from Lemma~\ref{lemma:torsione1} that the domain is symmetric with respect to $\{x=0\}$ and $\{y_j=0\}$ for all $j=1,\dots,N$ and by the well known result by~\cite{gnn}, the domain cannot be convex otherwise every solution of problem~\ref{PB1} has exactly one critical point in contradiction with Lemma~\ref{lemma:torsione2}.
\end{remark}

\subsection{Curvature of the domain}

In this section we prove that the domain $\Om_\e$ previously defined has {\em positive mean curvature}.

Let us start by a technical lemma that gives us an explicit formula to compute the mean curvature for manifolds which are preimage of a regular value of real functions.
The proof is postponed to the Appendix.
\begin{lemma}
\label{lemma:curv1} 
Let $\Sigma=F^{-1}(0)$, for some $F\in\mathcal C^2(\R\times\R^N,\R)$. Assume $0$ is a regular value for $F$ and $F_{y_iy_j}=0$ for all $i\not=j$. Then the \emph{mean curvature} of $\Sigma$ is given by
\[
K_m=-\frac{1}{N\abs{\nabla F}^3}\left[\sum_{j=1}^{N}\left(F_{x}^2F_{y_jy_j}-2F_{x}F_{y_j}F_{xy_j}+F_{y_j}^2F_{xx}\right)+\sum_{j=1}^NF_{y_j}^2\sum_{\substack{\ell=1\\\ell\not=j}}^NF_{y_\ell y_\ell}\right].
\]
\end{lemma}

Finally, we are able to compute the mean curvature of the boundary of the domain.

\begin{lemma}
\label{lemma:torsione5}
The mean curvature of the boundary of $\Omega_\e$ is strictly positive everywhere.
\end{lemma}
\begin{proof} 
We will apply the previous lemma to $F(x,y)=u_\e(x,y)$. Note that $\nabla u_\e\not=0$ on $\partial\Omega_\e$ from $(iv)$ of Lemma~\ref{lemma:torsione1}.
Let $(x^\e,y^\e)\in\partial\Omega_\e$ and from the asymptotic behavior of the derivatives of $v(t,y)$ for $t\to\infty$ we have
\begin{align*}
v_t&=-2kt^{2k-1}(1+o(1)),&v_s&=c_kt^{2k-2}s(1+o(1)),\\
v_{tt}&=-2k(2k-1)t^{2k-2}(1+o(1)),&v_{ts}&=c'_kt^{2k-3}s(1+o(1)),\\
v_{ss}&=c_kt^{2k-2}(1+o(1)),
\end{align*}
and from the estimate $\abs{x^\eps}\le\eps^{-\frac{1}{2k}}$ we get that for all $j=1,\dots,N$ the following quantities
\[
\eps v_{t}(x^\eps,y_j^\eps),\quad\eps v_{s}(x^\eps,y_j^\eps),\quad\eps v_{tt}(x^\eps,y_j^\eps),\quad\eps v_{ts}(x^\eps,y_j^\eps),\quad\eps v_{ss}(x^\eps,y_j^\eps),
\]
go to $0$ as $\eps\to0$.

Then we proceed by considering the cases $\abs{y^\eps}\not\to0$ and $\abs{y^\eps}\to0$.

\emph{Case $\abs{y^\eps}\not\to0$.}\\
We point out that for $\eps$ small enough there exists $j\in\set{1,\dots,N}$ such that $\partial_{y_j}u_\eps\not=0$, otherwise $\abs{y^\eps}\to0$.
Then from Lemma~\ref{lemma:curv1} we have
\[
K_m=-\frac{-(N-1)\abs{y^\eps}^2(1+o(1))}{N\left(\abs{y^\eps}^2(1+o(1))\right)^{\frac{3}{2}}}=\frac{N-1}{N\abs{y^\eps}}(1+o(1))>0.
\]
Not that the assumption $N\ge2$ is crucial. Indeed if $N=1$ the curvature changes sign, see~\cite{gg3}.

\emph{Case $\abs{y^\eps}\to0$.}\\
In this case, by~\ref{lemma:curv2:eq1} we have that $x^\e\to+\infty$ and for all $j=1,\dots,N$ fixed $\partial_{y_j}u_\e=o(1)$. Recalling~\ref{lemma:curv2:eq1} again, the following estimates hold true
\begin{align*}
(\partial_{y_j}u_\eps)^2\partial_{xx}u_\eps&=o(\e^{1-\frac{2k-2}{2k}})=o(\e^\frac{1}{k}),\\
\partial_{x}u_\eps\partial_{y_j}u_\eps\partial_{xy_j}u_\eps&=o\left(\eps^{1-\frac{2k-1}{2k}}\eps^{1-\frac{2k-3}{2k}}\right)=o(\e^\frac{2}{k})=o(\e^\frac{1}{k}),\\
(\partial_{x}u_\eps)^2\partial_{y_jy_j}u_\eps&=-\left(-2Nk\eps(x^\eps)^{2k-1}\right)^2(1+o(1))\\
&=-2^\frac{1}{k}N^2k^2\e^{\frac{1}{k}}(1+o(1)).
\end{align*}
This yields
\begin{equation}
\label{contoparziale:curv1}
(\partial_{y_j}u_\eps)^2\partial_{xx}u_\eps-2\partial_{x}u_\eps\partial_{y_j}u_\eps\partial_{xy_j}u_\eps+(\partial_{x}u_\eps)^2\partial_{y_
jy_j}u_\eps=-2^\frac{1}{k}N^2k^2\e^{\frac{1}{k}}(1+o(1)).
\end{equation}
Moreover by similar computations
\begin{equation}
\label{contoparziale:curv3}
\sum_{j=1}^N(\partial_{y_j}u_\eps)^2\sum_{\substack{\iota=1\\\iota\not=j}}^{N}\partial_{y_\iota y_\iota}u_\eps=-(N-1)(1+o(1))\sum_{j=1}^N(\partial_{y_j}u_\eps)^2\le0.
\end{equation}
Finally, we can apply Lemma~\ref{lemma:curv1}, and putting together~\ref{contoparziale:curv1} and~\ref{contoparziale:curv3} we have
\begin{align*}
-N\abs{\nabla u_\eps}^3K_m&\le\sum_{j=1}^N\left((\partial_{y_j}u_\eps)^2\partial_{xx}u_\eps-2\partial_{x}u_\eps\partial_{y_j}u_\eps\partial_{xy_j}u_\eps+(\partial_{x}u_\eps)^2\partial_{y_jy_j}u_\eps\right)\\
	&=-2^\frac{1}{k}N^3k^2\e^{\frac{1}{k}}(1+o(1))<0,
\end{align*}
that is $K_m>0$.
\end{proof}

\subsection{Proof of Theorem~\ref{THM2}}
\begin{proof}
The claims follow from Lemma~\ref{lemma:torsione1}, Lemma~\ref{lemma:torsione2} and Lemma~\ref{lemma:torsione5} considering $u_\e/N$.
\end{proof}

\begin{remark}
It is also possible to treat the case $x=(x_1,\dots,x_M)\in\R^M$, with $M>1$, in such a way that the domain $\Omega_\eps$ grows in $M$ directions.
The proof works replacing the function $u_\e$ by the following one
\[
\tilde u_\eps(x,y)=\frac{1}{2}\sum_{j=1}^N\left(1-y_j^2\right)+\eps\sum_{i=1}^M\sum_{j=1}^Nv(x_i,y_j).
\]
The computations are very similar to the case $M=1$. It is not difficult to generalize Lemma~\ref{lemma:curv1} taking into account that $\partial_{x_ix_h}u_\e=0$ for all $i\not=h$.
\end{remark} 

\appendix
\section{}
\label{appendix}
Here we show that there exist coefficients $\alpha_i\in\R$ such that the function introduced in subsection \ref{ss1}
\[
F(t)=\sum_{i=1}^n\alpha_i\cosh(\sqrt{\mu_i}t),
\]
admits $k$ nondegenerate maxima points.
\begin{lemma}
\label{lemmacosh}
For $k\in\N$ fixed, there exists $n=n(k)\in\N$ and $\alpha_1,\dots,\alpha_n\in\R$ such that the function
\[
F(t)=\sum_{i=1}^n\alpha_i\cosh(\sqrt{\mu_i}t),
\]
admits $k$ nondegenerate maxima points for $\alpha_1=-1$.
\end{lemma}
\begin{proof}
Let $1<\tau_1<\dots<\tau_k$ . For some $n=n(k)\in\N$ consider a polynomial $P(t)=\sum_{j=1}^na_jt^j$ such that
\begin{gather*}
a_n=-1\\
P'(\tau_i)=0,\quad\forall i=1,\dots,k,\\
P''(\tau_i)<0,\quad\forall i=1,\dots,k.
\end{gather*}
Let $0<t_1<\dots<t_k$ be such that $\cosh(t_i)=\tau_i$ for all $i=1,\dots,k$ and define $h(t)=P(\cosh(t))$. Then we have
\[
h'(t_i)=0,\qquad h''(t_i)<0,
\]
that is $\tau_1,\dots,\tau_k$ are nondegenerate maximum point for $h$. Up to a constant, from the binomial formula it is easy to see that for all $m\in\N$
\[
(\cosh(t))^m=\sum_{\ell=1}^mc(m,\ell)\cosh(\ell t),
\]
for suitables $c(m,\ell)>0$, with $c(m,m)=1$. Finally, for $\delta=\frac{\mu_0}{8n}$ the function
\[
F(t)=\sum_{j=1}^na_j\sum_{\ell=1}^jc(j,\ell)\cosh(\delta\ell t)
\]
is the function we were looking for. We point out that from the choice of $\delta$, \ref{scelta:mu} is satisfied.
\end{proof}

Now we prove that the critical points of the function
\[
q(t)=-\prod_{\ell=1}^k(t^2-t_k^2),\quad\text{with }k\in\N,\,\,\, k\ge2\text{ and }0<t_1<\dots<t_k,
\]
are nondegenerate.

\begin{lemma}
\label{lemmanondegen}
Let $q(t)=-\prod_{\ell=1}^k(t^2-t_k^2)$ with $k\in\N$, $k\ge2$ and $0<t_1<\dots<t_k$. Then the critical points of $f$ are nondegenerate.
\end{lemma}
\begin{proof}
Let $k>2$ (the case $k=2$ is left to the reader).
A straightforward computation shows that $q'(0)=0$ and $q''(0)\not=0$. Now let $\tau\ne0$ be such that $q'(\tau)=0$. Of course $q(\tau)\not=0$ and
\[
0=q'(\tau)=-2\tau\sum_{\ell=1}^k\prod_{\substack{h=1\\h\not=\ell}}^k(\tau^2-t_h^2),
\]
Finally, one has
\begin{align*}
q''(\tau)&=-4\tau^2\sum_{\ell=1}^k\sum_{\substack{h=1\\h\not=\ell}}^k\prod_{\substack{m=1\\m\not=\ell\\m\not=h}}^k(\tau^2-t_m^2)\\\\
	&=-4\tau^2\sum_{\ell=1}^k\frac{1}{(\tau^2-t_\ell^2)}\sum_{\substack{h=1\\h\not=\ell}}^k\prod_{\substack{m=1\\m\not=h}}^k(\tau^2-t_m^2)\\
	&=-4\tau^2\sum_{\ell=1}^k\frac{1}{(\tau^2-t_\ell^2)}\left[\underbrace{\sum_{h=1}^k\prod_{\substack{m=1\\m\not=h}}^k(\tau^2-t_m^2)}_{=0\hbox{ since }q'(\tau)=0}-\prod_{\substack{m=1\\m\not=\ell}}^k(\tau^2-t_m^2)\right]\\
	&=4\tau^2\sum_{\ell=1}^k\frac{1}{(\tau^2-t_\ell^2)}\prod_{\substack{m=1\\m\not=\ell}}^k(\tau^2-t_m^2)\\
	&=-4\tau^2q(\tau)\sum_{\ell=1}^k\frac{1}{(\tau^2-t_\ell^2)^2}\not=0.\qedhere
\end{align*}
\end{proof}

The following is the proof of Lemma~\ref{lemma:curv1} from Section~\ref{SEZ4}.

\begin{proof}[Proof of Lemma~\ref{lemma:curv1} ]
Let $\Phi=\frac1{\abs{\nabla F}}$ and consider the normal field
\[
\mathbf N=-\Phi\cdot(F_{x},F_{y_1},\dots,F_{y_N}).
\]
Then the mean curvature of $\Sigma$ is given by
\[
K_m(p)=\frac{1}{N}\mathrm{tr}(d\mathbf N_p).
\]
Taking into account that
\begin{align*}
\Phi_{x}&=-\Phi^3\left(F_{x}F_{xx}+\sum_{j=1}^NF_{y_j}F_{xy_j}\right),\\
\Phi_{y_j}&=-\Phi^3\left(F_{x}F_{xy_j}+F_{y_j}F_{y_jy_j}\right),
\end{align*}
one has
\begin{align*}
-\mathrm{tr}(d\mathbf N_p)=&\,\,\Phi\Delta F+\Phi_{x}F_{x}+\sum_{j=1}^N\Phi_{y_j}F_{y_j}\\
	=&\,\,\Phi^3\left[\abs{\nabla F}^2\left(F_{xx}+\sum_{j=1}^NF_{y_jy_j}\right)-\left(F_{x}F_{xx}+\sum_{j=1}^NF_{y_j}F_{xy_j}\right)F_x\right.\\
	&-\left.\sum_{j=1}^N\left(F_{x}F_{xy_j}+F_{y_j}F_{y_jy_j}\right)F_{y_j}\right]\\
	=&\,\,\Phi^3\left[\sum_{j=1}^N\left(F_x^2F_{y_jy_j}-2F_xF_{y_j}F_{xy_j}+F_{y_j}^2F_{xx}\right)+\sum_{j=1}^NF_{y_j}^2\sum_{\substack{\ell=1\\\ell\not=j}}^NF_{y_\ell y_\ell}\right]
\end{align*}
which yields the claim.
\end{proof}

\bibliographystyle{alphaabbrv}
\bibliography{paper.bib}

\end{document}